\documentclass[11pt]{amsart}
\usepackage{amsfonts}
\usepackage{amscd}
\usepackage{amssymb}
\usepackage{graphics}

\usepackage{amsmath}

\usepackage[english]{babel}
\usepackage{hyperref}
\hypersetup{colorlinks,citecolor=blue}

\newcommand{\BG}{{\mathbb{G}}}

\newcommand{\gd}{\delta}

\newcommand{\gs}{\sigma}

\newcommand{\gO}{\Omega}

\newcommand{\gep}{\epsilon}
\newcommand{\gl}{\lambda}

\newcommand{\diag}{\text{diag}}
\newcommand{\SL}{\text{SL}}

\newtheorem{prop}{Proposition}[section]
\newtheorem{thm}[prop]{Theorem}
\newtheorem{lem}[prop]{Lemma}

\theoremstyle{definition}

\newtheorem{rem}[prop]{Remark}

\begin{document}
\author{Tsachik Gelander}


\date{\today}

\title{On fixed points and uniformly convex spaces}
\maketitle

\let\languagename\relax 

The purpose of this note is to present two elementary, but useful, facts concerning actions on uniformly convex spaces. We demonstrate how each of them can be used in an alternative proof of the triviality of the first $L_p$-cohomology of higher rank simple Lie groups, proved in \cite{BFGM}.

\medskip

Let $G$ be a locally compact group with a compact generating set $K\ni 1$, and let $X$ be a complete Busemann non-positively curved uniformly convex metric space.\footnote{E.g. a CAT(0) or a uniformly convex Banach space, for general definition and more examples see \cite{GKM}.}
Suppose that $G$ acts continuously by isometries on $X$ such that the displacement goes to infinity, i.e.
$$
 d_K(x)=\max_{k\in K}d(x,k\cdot x)\to\infty~\text{~when~}~x\to\infty,
$$
where $x\to\infty$ means that $x$ eventually leaves every ball in $X$.\footnote{For Banach spaces, this means that the associated linear action does not have asymptotically invariant vectors.} We keep these assumptions for the next two sections.

\section{Minimal displacement}
Let $m=\inf_{x\in X}d_K(x)$. We say that $x\in X$ has minimal $K$--displacement if $d_K(x)=m$.

\begin{lem}\label{lem:minimal}
The set $M$ of points with minimal $K$--displacement is non-empty closed convex and bounded.
\end{lem}

\begin{proof}
For any $d>m$ the set $C_d=\{x\in X:d_K(x)\leq d\}$ is clearly closed and non-empty, convex since $X$ is BNPC, and bounded since we assume $d_K\to\infty$. By \cite{GKM}, Lemma 2.2\footnote{For UC Banach spaces this also follows from week compactness of closed bounded convex sets.} (see also \cite{Monod}) the intersection $\cap_{d>m}C_d=M$ of bounded closed convex sets is non-empty.
\end{proof}

\begin{rem}
In case $X$ is a Banach space and $G$ is a non-compact simple Lie group over a local field, and $K\subset G$ consists of non-ellyptic elements, one can show that $M$ must be a single point.
\end{rem}

\section{Minimal average displacement}
Let $\gO$ be a compact set such that $KU\subset\gO$ for some open set $U\subset G$. For $x\in X$ let $E_\gO(x):=\int_{g\in\gO}d(x,g\cdot x)^2$ be the averaged squared displacement of $x$, and let $r=\inf_{x\in X}E_\gO(x)$.

\begin{prop}\label{prop:average}
The set $P$ of points $p\in X$ with $E(p)=r$ is non-empty closed convex and bounded. Furthermore, if $X$ is a Banach space, $P$ is a singleton.
\end{prop}


\begin{lem}\label{lem:linear}
Given $x_0\in X$, there is a constant $c>0$ such that $d(x,x_0)\geq\frac{1}{c}\Rightarrow E(x)> c d(x,x_0)$.\footnote{For linear isometric actions on UC Banach space, the lemma says that non-existance of asymptotically invariant vectors is equivalent to the existence of  $\gep>0$ such that $\int_\gO\|gv-v\|\geq\gep\| v\|$ for every vector $v$.}
\end{lem}

\begin{proof}
Since $\gO\supset U\cup kU$ for every $k\in K$, the lemma follows easily from \cite{GKM}, Lemma 2.3.
\end{proof}


\begin{proof}[Proof of \ref{prop:average}]
It follows that for every $d>r$ the non-empty set $A_d=\{ x\in X:E(x)\leq d\}$ is bounded. It is also closed and convex. Thus, as above, one derives that $P=A_r=\cap_{d>r}A_d$ is non-empty.

Suppose now that $X$ is a Banach space. If $P$ would contain two different points, $p_1,p_2$ then, by strict convexity of $d\mapsto d^2$ we would either have $E(\frac{p_1+p_2}{2})<r$ or that $gp_1-p_1=gp_2-p_2$ for all $g\in\gO$ which means that $p_2-p_1$ is an invariant vector for the associated linear action. In either way we get a contradiction.
\end{proof}

\section{Fixed point for higher rank simple Lie group on $L_p(\mu)$-spaces}

We now give a proof of following theorem from \cite{BFGM} based on the results above. The argument is similar to \cite{BFGM} but somewhat more optimal.

\begin{thm}[\cite{BFGM}]
Let $k$ be a local field and $\BG$ a simple $k$--algebraic group of $k$--rank $\geq 2$. Then $\BG(k)$ has property $F_{L_p(\mu)}$ for any $1<p<\infty$ and any $\gs$-finite Borel measure $\mu$.
\end{thm}

Consider first the case of $\BG=\SL_3$ with  $\SL_2$ imbedded in the left upper corner. It is a well known consequence of strong relative property $(T)$ and the Howe--Moore theorem that for any unitary representation of $\SL_3(k)$ on a Hilbert space $H$, the induced representation of $\SL_2(k)$ on the orthogonal to the space of invariant vectors, does not have asymptotically invariant vectors.

As explained in \cite{BFGM}, one extends such statements to general $L_p$--spaces, by the following reasoning.
Let $p\ne 2$. By a classical theorem of Banach, any linear isometry $U$ of $L_p(\mu)$ is of the form
$U(f)(x)=f(Tx)\big(\frac{d_*T\mu}{d\mu}\big)^{1/p}h(x)$ for some measure class preserving transformation $T$ and
a function $h$ of absolute value 1. It follows that the Mazur map
$M_{2,p}:f\mapsto\text{sign}(f)|f|^\frac{2}{p}$ intertwines linear isometric actions on $L_p$ with isometric
actions on $L_2$. Moreover, the Mazur map and its inverse are uniformly continuous on bounded sets. We thus
obtain:

\begin{prop}
For any linear isometric representation of $\SL_3(k)$ on $B=L_p(\mu)$ the induced action of $\SL_2(k)$ on the quotient space $B/B'$, where $B'$ is the subspace of invariant vectors, does not have asymptotically invariant vectors.
\end{prop}


Suppose now that $\SL_3(k)$ acts by isometries on $B=L_p(\mu)$ and let $B'$ be the subspace of invariant vectors for the associated linear action. It follows that the action of $\SL_2(k)$ on $B/B'$ satisfies the assumptions of Sections 1,2. Fix a compact generating set $\gO$ for $\SL_2(k)$ and let $p_1$ be either the circumcenter of the set $M$ from Lemma \ref{lem:minimal} or the unique minimizing point for the average displacement, Proposition \ref{prop:average}. $p_1$ must be fixed by the centralizer
$A_1=\{\diag(\gl,\gl,\gl^{-2}):\gl\in k\}$ of $\SL_2(k)$. Considering the $\SL_2(k)$ on the right lower corner, one gets a fixed point $p_2$ for $A_2=\{\diag(\gd^{-2},\gd,\gd):\gd\in k\}$. By the Cartan decomposition, there is a compact group $K\leq\SL_3(k)$ such that $\SL_3(k)=KA^+K\subset KA_1A_2K$.  Since $K$ is compact it also has a fixed point.
Thus the $G~(=KA_1A_2K)-$action has bounded orbits, hence a global fixed point.

Finally, note that the exact same argument holds with $SL_3(k)$ replaced by $\text{Sp}_4(k)$, and the general case follows as any higher rank group contains either a copy of $\SL_3$ or $\text{Sp}_4$. Again, we refer to \cite{BFGM} for details.\qed

\begin{rem}
U. Bader suggested a third approach, similar to the one from Section 2; Suppose $\gO$ is a relatively compact open generating set. Using the Radon--Nikodym property of UC Banach space one can average along $\gO$ orbits, and obtain an affine operator $x\mapsto\int_\gO g\cdot x$. Using a variant of Lemma \ref{lem:linear} one shows that, in case there are no asymptotically invariant vectors, this operator has Lipschitz constant $<1$, and hence a unique fixed point.
\end{rem}



\begin{thebibliography}{99}
\bibitem{BFGM} U. Bader, A. Furman, T. Gelander, N. Monod, {\it Property (T) and rigidity for actions on Banach spaces}.  Acta Mathematica. 198 (2007), 57-105.

\bibitem{GKM} T. Gelander, A. Karlsson, G.A. Margulis, {\it Superrigidity, generalized harmonic maps and uniformly convex spaces}.
Geom. Funct. Anal. (to appear).

\bibitem{Monod} N.~Monod,
\newblock Superrigidity for irreducible lattices and geometric splitting.
\newblock {\em J. Amer. Math. Soc.}, \textbf{19} no.~4 (2006) 781--814.


\end{thebibliography}
\end{document}